\documentclass{amsart}
\usepackage{amssymb}
\usepackage{amsmath}
\usepackage{setspace,nth}
\usepackage{xcolor}
\newtheorem{theorem}{Theorem}[section]

\newtheorem{lemma}[theorem]{Lemma}

\theoremstyle{definition}

\newtheorem{conjecture}[theorem]{Conjecture}
\newtheorem{remark}[theorem]{Remark}

\numberwithin{equation}{section}

\begin{document}
\title[Moment of Kummer sums weighted by $L$-functions]
{Moment of Kummer sums weighted by $L$-functions}
\author{Nilanjan Bag}
 \address{Department of Mathematics, Ramakrishna Mission Vivekananda Educational and Research Institute, Belur, Howrah,  West Bengal-711202, India}
\email{nilanjanb2011@gmail.com}
\subjclass[2020]{11L05, 11L07.}
\date{18th January, 2023}
\keywords{generalized quadratic Gauss sums; Legendre symbol; asymptotic formula.}
\begin{abstract}
The main purpose of this article is to study higher order moments of Kummer sums weighted by $L$-functions using estimates for character sums and analytic methods. The results of this article complement a conjecture of Zhang Wenpeng (2002). Also the results in this article give analogous results of Kummer's conjecture (1846).
\end{abstract}
\maketitle
\section{Introduction and statements of the results}
Let $q\geq 2$ and $n$ be integers. For Dirichlet character $\chi\bmod q$, the generalized $k^\text{th}$ Gauss sums $G(n,k,\chi;q)$ are defined as
\begin{equation}
 G(n,k,\chi;q)=\sum_{a=1}^q\chi(a)e\left(\frac{na^k}{q}\right),\notag
\end{equation}
where $e(y)=e^{2\pi iy}$. The values of $G(n,k,\chi;q)$ behave irregularly whenever $\chi$ varies. For $k=2$, i.e. in the case of generalized quadratic Gauss sums, T. Cochrane, Z. Y. Zheng \cite{cochrane} proved that
\begin{align*}
|G(n,2,\chi;q)|\leq 2^{\omega (q)}\sqrt{q},
\end{align*}
where $\omega(q)$ denotes the number of distinct prime divisors of $q$. 
In case of prime $p$, finding such bounds is due to Weil \cite{weil, weil-2}. The Dirichlet $L$-function $L(s,\chi)$ corresponding to the character $\chi \bmod q$  is defined by 
\begin{align*}
L(s,\chi)=\sum_{n=1}^{\infty}\frac{\chi(n)}{n^s}.
\end{align*}
In analytic number theory many problems become very difficult when specialised to $L$-functions. Therefore it is a very important theme to study $L$-functions, specially higher moments at critical line.
\par In \cite{kummer}, Kummer studied the distribution of the cubic exponential sums
\begin{align*}
S_p=\sum_{a=1}^{p}e\left(\frac{a^3}{p}\right),
\end{align*}
for prime $p\equiv 1\pmod 3$. When weighted by some Dirichlet character $\chi$, we denote this sum as $S_p(n;\chi)$, which is defined as 
\begin{align*}
S_p(n;\chi)=\sum_{a=1}^p\chi(a)e\left(\frac{na^3}{p}\right).
\end{align*}
Hence $S_p(n;\chi)= G(n,3,\chi;p)$. It is well known that $S_p\leq 2\sqrt{p}$ and one can subsequently write 
\begin{align*}
\frac{S_p}{2\sqrt{p}}=\cos(2\pi \theta_p),
\end{align*}
where $\theta_p\in [0,1]$. To investigate the distribution of $\theta_p$ in $[0,1]$ Kummer computed the frequency with which $\cos(2\pi\theta_p)$ lie in the interval $[-1,-\frac{1}{2}]$, $[-\frac{1}{2}, \frac{1}{2}]$ and $[\frac{1}{2},1]$ upto $p\leq 500$. He found that $\frac{S_p}{2\sqrt{p}}$ lie with a frequency $1:2:3$ and he hesitantly conjectured that it may be true asymptotically. Later this conjecture was disproved by Heath-Brown and Patterson \cite{HB-P}. In light of the numerical evidence, Patterson \cite{pat} conjectured that
\begin{conjecture}[Patterson, 1978]
As $X\rightarrow\infty$, 
\begin{align*}
\sum_{\substack{p\leq X\\ p\equiv 1\bmod 3}}\frac{S_p}{2\sqrt{p}}\sim d\frac{X^{5/6}}{\log X}
\end{align*}
where $d=\frac{2(2\pi)^{2/3}}{5\Gamma(\frac{2}{3})}$ and $p$ runs through primes. Here $\Gamma$ is the classical gamma function.
\end{conjecture}
The goal of this article is to make similar study to Kummer sums when taken average over group of characters mod $p$. In particular, in this article we study second and fourth power mean values of Kummer sums and calculate moments when the sums are weighted by $L$-functions. To be specific, we prove the following theorems.
\begin{theorem}\label{mt1}
Let $p$ an odd prime such that $p\equiv 1\bmod 3$ and $n$ be an integer. Then for any Dirichlet character $\chi \bmod p$, we have an asymptotic formula
\begin{align*}
\sum_{\chi \bmod p}|S_p(n;\chi)|^{2}=p^2+O(p^{3/2}).
\end{align*}

\end{theorem}
Also we get higher power mean values, which is given in the following theorem.
\begin{theorem}\label{mt2}
Let $p$ an odd prime such that $p\equiv 1\bmod 3$ and $n$ be an integer. Then for any Dirichlet character $\chi \bmod p$, we have an asymptotic formula
\begin{align*}
\sum_{\chi \bmod p}|S_p(n;\chi)|^{4}=5p^3+O(p^{5/2}).
\end{align*}
\end{theorem}
Deriving the above moments when weighted by $L$-functions becomes critical with higher level of difficulty. In this article we capture the second and fourth power moment of the Kummer sums weighted by $L$-functions.
First we fix a constant $C$, as 
\begin{align}\label{constant-c}
C=\prod_p\left[1+\frac{\binom{2}{1}^2}{4^2.p^2}+\frac{\binom{4}{2}^2}{4^4.p^4}+\cdots+\frac{\binom{2m}{m}^2}{4^{2m}.p^{2m}}+\cdots\right].
\end{align}
Let $r(n)$ be a multiplicative function defined by
\begin{align*}
r(p^{\alpha})=\frac{\binom{2\alpha}{\alpha}}{4^{\alpha}},
\end{align*}
with $r(1)=1$, where $\alpha$ is a positive integer and $p$ is a prime. For $1\leq t\leq p-1$ we fix 
\begin{align*}
C_t=\sum_{s|t}\frac{r(s)r(st)}{s^2t}.
\end{align*}
Consider the equation $x^3-1\equiv 0\bmod p$. This congruence has solutions apart from $1$, when $p\equiv 1 \bmod 3 $. We fix the solutions as $a_1$ and $a_2$, where $a_1,a_2\neq 1$. Notice that $a_1a_2\equiv 1 \mod p$.
Now we are ready to state the following two mixed moments.
\begin{theorem}\label{mt3}
Let $p$ an odd prime such that $p\equiv 1\bmod 3$ and $n$ be an integer. Then for any Dirichlet character $\chi \bmod p$, we have an asymptotic formula
\begin{align*}
\sum_{\chi\neq\chi_0}|S_p(n;\chi)|^{2}\cdot |L(1,\chi)|=C\cdot (1+C_{a_1}+C_{\overline{a}_1}) \cdot p^2+O(p^{3/2}\cdot \ln p).
\end{align*}
\end{theorem}
\begin{theorem}\label{mt4}
Let $p$ an odd prime such that $p\equiv 1\bmod 3$ and $n$ be an integer. Then for any Dirichlet character $\chi \bmod p$, we have an asymptotic formula
\begin{align*}
\sum_{\chi\neq\chi_0}|S_p(n;\chi)|^{4}\cdot |L(1,\chi)|= C\cdot \left(3+2(C_{a_1}+C_{\overline{a}_1})+C_{a_1^2}+C_{\overline{a}_1^2}\right)\cdot p^3\\
\hspace{1cm}+O(p^{5/2}\cdot \ln p).
\end{align*}
\end{theorem}
\begin{remark}
Notice that for $p\equiv 2 \bmod 3$, the map 
\begin{align*}
f:\mathbb{Z}^{\times}_p\rightarrow\mathbb{Z}_p^{\times} ~\text{defined by}~  x\mapsto x^3.
\end{align*}
is a bijection. Hence it is sufficient to do the above study only for prime $p\equiv 1 \bmod 3.$
\end{remark}
\section{notations and preliminaries}
\begin{itemize}
\item Let $\chi_0$ denotes the principal character modulo $p$, defined as $\chi_0(a)=1$ for $(a,p)=1$ and $\chi(a)=0$ otherwise.

\item We denote the classical Gauss sum as $g(\chi)$ for a character $\chi$ and it is defined as
\begin{align*}
g(\chi)=\sum_{a=1}^{p}\chi(a)e(\frac{a}{p}).
\end{align*} 
\item For characters $\chi$ and $\psi$, $J(\chi,\psi)$ denotes the classical Jacobi sum which is defined as 
\begin{align*}
J(\chi,\psi)=\sum_{a=1}^{p}\chi(a)\psi(a-1).
\end{align*}
\item For every prime $p\equiv 1 \bmod 3$, let $\lambda$ be the cubic residue symbol, defined as 
\begin{align*}
\lambda(a)=\left(\frac{a}{p}\right)_3\equiv a^{(p-1)/3} \bmod p.
\end{align*}
\item The notation $f=O(g)$ or $f\ll g$ implies $f\leq c\cdot g$, where the $c$ is a constant which sometimes may depends upon some $\epsilon$.
\item$\overline{x}$ denotes the multiplicative inverse of $x$ modulo some prime $p$.

\end{itemize}
We also know that for $\chi$ and $\psi$ non-trivial one can get
\begin{align*}
|g(\chi)|=|J(\chi,\psi)|=\sqrt{p}.
\end{align*}. 
\par Let $N=p^{3/2}$ and $\chi$ be a non-principal multiplicative character modulo $p$. Then from \cite[equation (1)]{zhang} we get 
\begin{align}\label{l(1,X)}
\sum_{\chi\neq \chi_0}|L(1,\chi)|&=\sum_{\chi\neq \chi_0}\left|\sum_{n\leq N}\frac{\chi(n)}{n}\right|+O(\ln p).
\end{align}
 From \cite[equation (2)]{zhang} one can directly have
\begin{align}\label{g2}
\left[\sum_{n\leq N}\frac{\chi(n)r(n)}{n}\right]^2=\sum_{n\leq N}\frac{\chi(n)}{n}+\sum_{N<n\leq N^2}\frac{\chi(n)r(n,N)}{n},
\end{align}
where $r(n,N)=\sum_{d|n, d,n/d\leq N}r(d)r\left(\frac{n}{d}\right)$.
\par $\bullet$ From \eqref{constant-c}, we have
\begin{align*}
C=\sum_{n=1}^{\infty}\frac{r(n)^2}{n^2}.
\end{align*}

\section{Some lemmas}
 First we prove a few lemmas which play a very important part in the proof of our main theorems.
\begin{lemma}\label{lem-1}
Let $p$ be an odd prime such that $p\equiv 1\bmod 3$ and $n$ be any integer with $\gcd(n,p)=1$. Then for any non-principal character  $\chi$ modulo $p$ the following identity holds
 \begin{equation*}
  |S_p(n;\chi)|^2=p(1+\chi(a_1)+\chi(a_2))+\sum_{\substack{a=2\\a\neq a_1, a_2}}^{p-1}\chi(a)\left(g(\lambda)\lambda(\overline{x})+g(\lambda^{2})\lambda^2(\overline{x})\right).
   \end{equation*}
 If $\chi_0$ is the principal character modulo $p$, then
 \begin{align*}
 |S_p(n;\chi_0)|^2&=3-\sum_{i=1}^{2}\left(g(\lambda)\overline{\lambda}(n(a_i-1))+g(\lambda^2)\overline{\lambda}^2(n(a_i-1))\right)\left(1+\lambda(a_i)+\lambda(a_i)^{2}\right)\\&+\overline{\lambda}(n)g(\lambda)\left(J(\overline{\lambda},\lambda)+J(\overline{\lambda},\lambda^2)\right)+\overline{\lambda}^2(n)g(\lambda^2)\left(J(\overline{\lambda}^2,\lambda)+J(\overline{\lambda}^2,\lambda^2)\right).
 \end{align*}
\end{lemma}
 
\begin{proof}
First, let $\chi$ be non-principal. We have
\begin{align*}
|S_p(n;\chi)|^2&=\sum_{a=1}^{p-1}\sum_{b=1}^{p-1}\chi(a)\overline{\chi}(b)e\left(\frac{na^3-nb^3}{p}\right)\\
&=\sum_{a=1}^{p-1}\sum_{b=1}^{p-1}\chi(a)e\left(\frac{nb^3(a^3-1)}{p}\right)\\
&=\sum_{a=1}^{p-1}\chi(a)\sum_{b=1}^{p}e\left(\frac{nb^3(a^3-1)}{p}\right)-\sum_{a=1}^{p-1}\chi(a)\\
&=p(1+\chi(a_1)+\chi(a_2))+\sum_{\substack{a=2\\a\neq a_1, a_2}}^{p-1}\chi(a)\sum_{b=1}^{p}e\left(\frac{nb^3(a^3-1)}{p}\right),
\end{align*}
where $a_1$ and $a_2$ are same as in Section 2.
Let $n(a^3-1)=x_a$, then the above expression becomes 
\begin{align*}
|S_p(n;\chi)|^2=p(1+\chi(a_1)+\chi(a_2))+\sum_{\substack{a=2\\a\neq  a_1, a_2}}^{p-1}\chi(a)\sum_{b=1}^{p}e\left(\frac{x_ab^3}{p}\right).
\end{align*}

Notice that inverse of $x_a$ exists for $a\neq 1,a_1,a_2$. Using cubic residue symbol $\lambda$, we can re-write the above expression as 
\begin{align*}
|S_p(n;\chi)|^2&=p(1+\chi(a_1)+\chi(a_2))+\sum_{\substack{a=2\\a\neq  a_1,a_2}}^{p-1}\chi(a)\sum_{b=1}^{p}e\left(\frac{x_ab}{p}\right)\left(1+\lambda(b)+\lambda(b)^{2}\right)\\
&=p(1+\chi(a_1)+\chi(a_2))+\sum_{\substack{a=2\\a\neq a_1, a_2}}^{p-1}\chi(a)\sum_{b=1}^{p}e\left(\frac{b}{p}\right)\left(1+\lambda(b\overline{x_a})+\lambda(b\overline{x_a})^{2}\right)\\
&=p(1+\chi(a_1)+\chi(a_2))\\&\qquad+\sum_{\substack{a=2\\a\neq  a_1,a_2}}^{p-1}\chi(a)\left(\sum_{b=1}^{p}e\left(\frac{b}{p}\right)+\lambda(\overline{x_a})\sum_{b=1}^{p}e\left(\frac{b}{p}\right)\lambda(b)+\lambda(\overline{x_a})^{2}\sum_{b=1}^{p}e\left(\frac{b}{p}\right)\lambda(b)^{2}\right)\\
&=p(1+\chi(a_1)+\chi(a_2))+\sum_{\substack{a=2\\a\neq a_1,a_2}}^{p-1}\chi(a)\left(g(\lambda)\lambda(\overline{x_a})+g(\lambda^{2})\lambda^2(\overline{x_a})\right),
\end{align*} 
where $g$ is the classical Gauss sum as defines in Section 2.
\par Now let $\chi=\chi_0$. Then 
\begin{align*}
|S_p(n;\chi_0)|^2=\sum_{a=1}^{p-1}\sum_{b=1}^{p-1}e\left(\frac{na^3-nb^3}{p}\right)
&=\sum_{a=1}^{p-1}\sum_{b=1}^{p-1}e\left(\frac{nb^3(a^3-1)}{p}\right)\\
&=3+\sum_{\substack{a=2\\a\neq a_1,a_2}}^{p-1}\sum_{b=1}^{p}e\left(\frac{nb^3(a^3-1)}{p}\right).
\end{align*}
Then similarly using cubic residue symbol, one can have 
\begin{align*}
|S_p(n;\chi_0)|^2&=3+\sum_{\substack{a=2\\a\neq  a_1,a_2}}^{p-1}\sum_{b=1}^{p}e\left(\frac{bx_a}{p}\right)\left(1+\lambda(b)+\lambda(b)^{2}\right)\\
&=3+\sum_{\substack{a=2\\a\neq a_1,a_2}}^{p-1}\sum_{b=1}^{p}e\left(\frac{bx_a}{p}\right)\left(1+\lambda(b)+\lambda(b)^{2}\right)\\
&=3+\sum_{\substack{a=2\\a\neq a_1,a_2}}^{p-1}\left(g(\lambda)\lambda(\overline{x_a})+g(\lambda^{2})\lambda^2(\overline{x_a})\right)\\
&=3+\sum_{\substack{a=2\\a\neq a_1,a_2}}^{p-1}\left(g(\lambda)\overline{\lambda}(n(a-1))+g(\lambda^{2}){\overline{\lambda}}^2(n(a-1))\right)\left(1+\lambda(a)+\lambda(a)^{2}\right)\\
&=3-\sum_{i=1}^{2}\left(g(\lambda)\overline{\lambda}(n(a_i-1))+g(\lambda^2)\overline{\lambda}^2(n(a_i-1))\right)\left(1+\lambda(a_i)+\lambda(a_i)^{2}\right)\\&+\left(\overline{\lambda}(n)g(\lambda)J(\overline{\lambda},\lambda)+\overline{\lambda}(n)g(\lambda)J(\overline{\lambda},\lambda^2)+\overline{\lambda}^2(n)g(\lambda^2)J(\overline{\lambda}^2,\lambda)\right.\\&\hspace{7cm}\left.+\overline{\lambda}^2(n)g(\lambda^2)J(\overline{\lambda}^2,\lambda^2)\right).
\end{align*}
This completes the proof the lemma.
\end{proof}
The next two lemmas are due to Zhang \cite{zhang}, which play an integral part in the proofs of our main theorems.

\begin{lemma}[Lemma 1, \cite{zhang}]\label{lem-2}For any odd prime $p$, we have the estimate

\begin{align*}
\sum_{a=1}^{p-1}\left|\sum_{\chi\neq \chi_0}\chi(a)|L(1,\chi)|\right|=O(p\cdot\ln p).
\end{align*}
\end{lemma}
\begin{lemma}[Lemma 2, \cite{zhang}]\label{lem-3}
For any odd prime $p$, we have the asymptotic formula
\begin{align*}
\sum_{{\substack{\chi(-1)=1\\\chi\neq \chi_0}}}|L(1,\chi)|=\frac{1}{2}\cdot C\cdot p+O(p^{1/2}\cdot\ln p),
\end{align*}
where $C$ is the constant defined in Section 2.
\end{lemma}
In the next lemma we extend the result in the last lemma to summation over all characters modulo $p$. The proof follows similar techniques of the proof of \cite[Lemma 2]{zhang}. 
\begin{lemma}\label{lem-4}
For any odd prime $p$, we have the asymptotic formula
\begin{align*}
\sum_{\chi\neq \chi_0}|L(1,\chi)|=C\cdot p+O(p^{1/2}\cdot \ln p),
\end{align*}
where $C$ is the constant defined in Section 2.
\end{lemma}
\begin{proof}
Take $N=p^{3/2}$. Using \eqref{l(1,X)}, we have
\begin{align*}
\sum_{\chi\neq \chi_0}|L(1,\chi)|&=\sum_{\chi\neq \chi_0}\left|\sum_{n\leq N}\frac{\chi(n)}{n}\right|+O(\ln p)\\
&=\sum_{\chi\neq \chi_0}\left|\sum_{n\leq N}\frac{\chi(n)r(n)}{n}\right|^2+O(\ln p)\\
&\qquad+\sum_{\chi\neq\chi_0}\left(\left|\sum_{n\leq N}\frac{\chi(n)}{n}\right|-\left|\sum_{n\leq N}\frac{\chi(n)r(n)}{n}\right|^2\right).
\end{align*}
This using \eqref{g2}, \cite[equation (3)]{zhang} and H\"older inequality, we get 
\begin{align*}
\sum_{\chi\neq \chi_0}|L(1,\chi)|&=(p-1)\sum_{m\leq N}\sum_{n=  m\bmod p, n\leq N}\frac{r(m)r(n)}{mn}+O(\ln p)\\
&\qquad\qquad+O\left(\sum_{\chi\neq \chi_0}\left|\sum_{N<n\leq N^2}\frac{\chi(n)r(n,N)}{n}\right|\right)\\
&=p\sum_{\substack{n=1\\(n,p)=1}}^{\infty}\frac{r(n)^2}{n^2}\\
&\qquad+O(\ln p)+O\left(p^{1/2}\left(\sum_{\chi\neq \chi_0}\left|\sum_{N<n\leq N^2}\frac{\chi(n)r(n,N)}{n}\right|^2\right)^{1/2}\right)\\
&=p\sum_{n=1}^{\infty}\frac{r(n)^2}{n^2}+O(p^{1/2}\cdot\ln p)\\
&=C\cdot p+O(p^{1/2}\cdot \ln p).
\end{align*}

\end{proof}
Now we take weighted sum with some non-principal Dirichlet character $\chi$ mod $p$. In particular, we prove 
\begin{lemma}\label{lem-5}
For any odd prime $p$, we have the asymptotic formula
\begin{align*}
\sum_{\chi\neq \chi_0}\chi(t)|L(1,\chi)|= C\cdot C_{t}\cdot p+O(p^{1/2}\ln p).
\end{align*}
for any $t\neq 0, 1$; where 
$
C_t=\sum_{s|t}\frac{r(s)r(st)}{s^2t}.
$

\end{lemma}

\begin{proof}
\begin{align*}
\sum_{\chi\neq \chi_0}|L(1,\chi)|&=\sum_{\chi\neq \chi_0}\chi(t)\left|\sum_{n\leq N}\frac{\chi(n)}{n}\right|+O(\ln p)\\
&=\sum_{\chi\neq \chi_0}\chi(t)\left|\sum_{n\leq N}\frac{\chi(n)r(n)}{n}\right|^2+O(\ln p)\\
&\hspace{3cm}+\sum_{\chi\neq \chi_0}\left(\left|\sum_{n\leq N}\frac{\chi(n)}{n}\right|-\left|\sum_{n\leq N}\frac{\chi(n)r(n)}{n}\right|^2\right)\\
\end{align*}
Let $(n,t)=s$ and we write $n=n's$. Then we have $(n',t)=1$. Hence the first term can be simplified more as 
\begin{align*}
&\sum_{\chi\neq \chi_0}\chi(t)\left|\sum_{n\leq N}\frac{\chi(n)r(n)}{n}\right|^2\\
&=(p-1)\sum_{\substack{n\leq N\\(n,p)=1}}\frac{r(n)r(nt)}{n^2t}\\
&=(p-1)\sum_{\substack{n=1\\(n,p)=1}}^{\infty}\frac{r(n)r(nt)}{n^2t}+O\left(\sum_{\substack{n\geq N}}^{\infty}\frac{r(n)r(nt)}{n^2t}\right)\\
&=(p-1)\sum_{s|t}\sum_{\substack{n=1\\n=n's\\(n',t)=1\\(n',p)=1}}^{\infty}\frac{r(n)r(nt)}{n^2t}+O\left(\ln^2 p\right)\\
&=(p-1)\sum_{s|t}\sum_{\substack{n=1\\n=n's\\(n',t)=1\\(n',p)=1}}^{\infty}\frac{r(n's)r(n'st)}{{n'}^2s^2t}+O\left(\ln^2 p\right)\\
&=(p-1)\sum_{s|t}\frac{r(s)r(st)}{s^2t}\sum_{\substack{(n',p)=1}}^{\infty}\frac{r(n')^2}{{n'}^2}+O\left(\ln^2 p\right)\\
&=(p-1)\sum_{s|t}\frac{r(s)r(st)}{s^2t}\sum_{\substack{n'=1\\(n',p)=1}}^{\infty}\frac{r(n')^2}{{n'}^2}+O\left(\ln^2 p\right)\\
&=C\cdot C_t\cdot p+O(\ln^2 p),
\end{align*}
where 
\begin{align*}
C_t=\sum_{s|t}\frac{r(s)r(st)}{s^2t}.
\end{align*}

\end{proof}
\section{Proof of Theorem \ref{mt1} and \ref{mt2}}
First using Lemma \ref{lem-1}, we have 
\begin{align*}
\sum_{\chi\bmod p}|S_p(\chi)|^2&=\sum_{\chi\neq \chi_0 }|S_p(\chi)|^2+|S_p(\chi_0)|^2\\&=p\sum_{\chi\neq \chi_0}\left(1+\chi(a_1)+\chi(a_2)\right)+\\&\sum_{\substack{a=2\\a\neq a_1,a_2}}^{p-1}\left(g(\lambda)\lambda(\overline{x})+g(\lambda^{2})\lambda(\overline{x})^{2}\right)\left(\sum_{\chi\bmod p}\chi(a)-1\right)\\&+3-\sum_{i=1}^{2}\left(g(\lambda)\overline{\lambda}(n(a_i-1))+g(\lambda^2)\overline{\lambda}^2(n(a_i-1))\right)\left(1+\lambda(a_i)+\lambda(a_i)^{2}\right)\\&+\overline{\lambda}(n)g(\lambda)\left(J(\overline{\lambda},\lambda)+J(\overline{\lambda},\lambda^2)\right)+\overline{\lambda}^2(n)g(\lambda^2)\left(J(\overline{\lambda}^2,\lambda)+J(\overline{\lambda}^2,\lambda^2)\right)
\end{align*}
Hence applying orthogonality properties of characters and using the estimates for classical Gauss sums and Jacobi sums one can directly have 
\begin{align*}
\sum_{\chi\bmod p}|S_p(\chi)|^2=p^2+O(p^{3/2}).
\end{align*}
This completes the proof of Theorem \ref{mt1}.
\par Next using estimates for Gauss and Jacobi sums one can write,
\begin{align}\label{p1}
\sum_{\chi\bmod p}|S_p(\chi)|^4&=\sum_{\chi\neq \chi_0 }|S_p(\chi)|^4+|S_p(\chi_0)|^4\\&=\sum_{\chi\neq \chi_0}\left[p(1+\chi(a_1)+\chi(a_2))+\sum_{\substack{a=2\\a\neq a_1, a_2}}^{p-1}\chi(a)\sum_{b=1}^{p}e\left(\frac{x_ab^3}{p}\right)\right]^2+O(p^2)\notag\\
&=A_1+A_2+A_3+O(p^2),
\end{align}
where 
\begin{align*}
A_1&=p^2\sum_{\chi\neq  \chi_0}(1+\chi(a_1)+\chi(a_2))^2,\\
A_2&=2p\sum_{\chi\neq \chi_0}(1+\chi(a_1)+\chi(a_2))\sum_{\substack{a=2\\a\neq a_1, a_2}}^{p-1}\chi(a)\sum_{b=1}^{p}e\left(\frac{xb^3}{p}\right),\\
A_3&=\sum_{\chi\neq  \chi_0}\left(\sum_{\substack{a=2\\a\neq a_1, a_2}}^{p-1}\chi(a)\sum_{b=1}^{p}e\left(\frac{x_ab^3}{p}\right)\right)^2.
\end{align*}
$A_1$ and $A_2$ are easy to evaluate.  
\begin{align*}
A_1&=p^2\sum_{\chi\neq \chi_0}(1+\chi(a_1)+\chi(a_2))^2\\
&=p^2\sum_{\chi\neq \chi_0}\left(1+2(\chi(a_1)+\chi(a_2))+\chi^2(a_1)+\chi^2(a_2)+2\chi(a_1a_2)\right)\\&=3p^3+O(p^2),
\end{align*}
as $a_1a_2\equiv 1 \mod p$.
Also
\begin{align*}
A_2&=2p\sum_{\chi\neq \chi_0}(1+\chi(a_1)+\chi(a_2))\sum_{\substack{a=2\\a\neq a_1,a_2}}^{p-1}\chi(a)\sum_{b=1}^{p}e\left(\frac{xb^3}{p}\right)\\
&=2p\sum_{\substack{a=2\\a\neq a_1, a_2}}^{p-1}\sum_{b=1}^{p}e\left(\frac{xb^3}{p}\right)\left(\sum_{\chi\bmod p}\chi(a)-1\right)\\
&+2p\sum_{\substack{a=2\\a\neq a_1, a_2}}^{p-1}\sum_{b=1}^{p}e\left(\frac{xb^3}{p}\right)\left(\sum_{\chi\bmod p}\chi(aa_1)-1\right)\\
&+2p\sum_{\substack{a=2\\a\neq a_1, a_2}}^{p-1}\sum_{b=1}^{p}e\left(\frac{xb^3}{p}\right)\left(\sum_{\chi\bmod p}\chi(aa_2)-1\right)\\
&=O(p^{5/2}).
\end{align*}
To get an asymptotic for $A_3$, we write the sum as 
\begin{align*}
A_3&=\sum_{\chi\neq  \chi_0}\sum_{\substack{a=2\\a\neq a_1\\a\neq a_2}}^{p-1}\sum_{\substack{c=2\\c\neq a_1\\c\neq a_2}}^{p-1}\chi(a\overline{c})\sum_{b=1}^{p}\sum_{d=1}^{p}e\left(\frac{x_ab^3-x_cd^3}{p}\right),
\end{align*}
where $x_a=n(a^3-1)$ and $x_c=n(c^3-1)$. Then using the orthogonality property of characters and properties of cubic residue symbol, we have 
\begin{align*}
A_3&=\sum_{\chi\neq  \chi_0}\sum_{\substack{a=2\\a\neq a_1\\a\neq a_2}}^{p-1}\sum_{\substack{c=2\\c\neq a_1\\c\neq a_2}}^{p-1}\chi(a\overline{c})\sum_{b=1}^{p}\sum_{d=1}^{p}e\left(\frac{x_ab^3-x_cd^3}{p}\right)\\
&=\sum_{\substack{a=2\\a\neq a_1\\a\neq a_2}}^{p-1}\sum_{\substack{c=2\\c\neq a_1\\c\neq a_2}}^{p-1}\sum_{b=1}^{p}\sum_{d=1}^{p}e\left(\frac{x_ab^3-x_cd^3}{p}\right)\left(\sum_{\chi\bmod p}\chi(a\overline{c})-1\right)\\
&=A_3^{'}-A_3^{''},
\end{align*}
where
\begin{align*}
A_3^{'}&=(p-1)\sum_{\substack{a=2\\a\neq a_1,a_2}}^{p-1}\sum_{b=1}^{p}\sum_{d=1}^{p}e\left(\frac{xb^3-xd^3}{p}\right)'\\
A_3^{''}&=\sum_{\substack{a=2\\a\neq a_1\\a\neq a_2}}^{p-1}\sum_{\substack{c=2\\c\neq a_1\\c\neq a_2}}^{p-1}\sum_{b=1}^{p}\sum_{d=1}^{p}e\left(\frac{x_ab^3-x_cd^3}{p}\right)
\end{align*}
Now
\begin{align*}
A_3^{'}
&=(p-1)\sum_{\substack{a=2\\a\neq a_1, a_2}}^{p-1}\sum_{b=1}^{p}\sum_{d=1}^{p}e\left(\frac{x_ab^3-x_ad^3}{p}\right)\\
&=(p-1)\sum_{\substack{a=2\\a\neq a_1,a_2}}^{p-1}\sum_{b=1}^{p}\sum_{d=1}^{p}e\left(\frac{b-d}{p}\right)\left(1+\lambda(\overline{x_a})\lambda(b)+\lambda^2(\overline{x_a})\lambda^{2}(b)\right)\\&\qquad\qquad\qquad\qquad\qquad\qquad\qquad\qquad\qquad\left(1+\lambda(\overline{x_a})\lambda(d)+\lambda^2(\overline{x_a})\lambda^{2}(d)\right)\\
&=(p-1)\sum_{b=1}^{p}\sum_{d=1}^{p}e\left(\frac{b-d}{p}\right)\sum_{\substack{a=2\\a\neq a_1, a_2}}^{p-1}\sum_{i=0}^{4}\lambda^i(\overline{x}_a)\sum_{j_1=0}^{2}\sum_{\substack{j_2=0\\j_1+j_2=i}}^{2}\lambda^{j_1}(b)\lambda^{j_2}(d)\\
&=(p-1)\sum_{b=1}^{p}\sum_{d=1}^{p}e\left(\frac{b-d}{p}\right)\sum_{i=0}^{4}\sum_{j_1=0}^{2}\sum_{\substack{j_2=0\\j_1+j_2=i}}^{2}\lambda^{j_1}(b)\lambda^{j_2}(d)\sum_{\substack{a=2\\a\neq a_1, a_2}}^{p-1}\overline{\lambda}(n^i(a^3-1)^i)\\
&=(p-1)\sum_{b=1}^{p}\sum_{d=1}^{p}e\left(\frac{b-d}{p}\right)\sum_{i=0}^{4}\sum_{j_1=0}^{2}\sum_{\substack{j_2=0\\j_1+j_2=i}}^{2}\lambda^{j_1}(b)\lambda^{j_2}(d)\sum_{\substack{a=2\\a\neq a_1, a_2}}^{p-1}\overline{\lambda}(n^i(a-1)^i)\\
&\qquad\qquad\qquad\qquad\qquad\qquad\qquad\qquad\qquad\qquad\qquad\qquad(1+\lambda(a)+\lambda^{2}(a))\\
&=(p-1)\sum_{b=1}^{p}\sum_{d=1}^{p}e\left(\frac{b-d}{p}\right)\sum_{i=0}^{4}\overline{\lambda}^i(n){\sum_{j_1=0}^{2}}\lambda^{j_1}(b)\lambda^{i-j_1}(d)\sum_{r=0}^{2}J(\overline{\lambda}^i,\lambda^r)+O(p^{5/2})\\
&=(p-1)\sum_{i=0}^{4}\overline{\lambda}^i(n)\sum_{r=0}^{2}J(\overline{\lambda}^i,\lambda^r){\sum_{j_1=0}^{2}}\sum_{b=1}^{p}\sum_{d=1}^{p}e\left(\frac{b-d}{p}\right)\lambda^{j_1}(b)\lambda^{i-j_1}(d)+O(p^{5/2})\\
&=(p-1)\sum_{i=0}^{4}\overline{\lambda}^i(n)\sum_{r=0}^{2}J(\overline{\lambda}^i,\lambda^r){\sum_{j_1=0}^{2}}\lambda^{i-j_1}(-1)g(\lambda^{j_1})g(\lambda^{i-j_1})+O(p^{5/2})\\
\end{align*} 
where $g$ is the classical gauss sums. For a non- principal Dirichlet character $\chi$, we know that 
\begin{align*}
g(\chi)g(\overline{\chi})=\chi(-1)p.
\end{align*} 
Using the above equality, we can break the expression in two parts as
\begin{align*}
&\sum_{i=0}^{4}\overline{\lambda}^i(n)\sum_{r=0}^{2}J(\overline{\lambda}^i,\lambda^r){\sum_{j_1=0}^{2}}\lambda^{i-j_1}(-1)g(\lambda^{j_1})g(\lambda^{i-j_1})\\
&=\sum_{r=0}^{2}J(\chi_0,\lambda^r){\sum_{j_1=0}^{2}}\lambda^{-j_1}(-1)g(\lambda^{j_1})g(\lambda^{-j_1})\\
&+\sum_{i=1}^{4}\overline{\lambda}^i(n)\sum_{r=0}^{2}J(\overline{\lambda}^i,\lambda^r){\sum_{j_1=0}^{2}}\lambda^{i-j_1}(-1)g(\lambda^{j_1})g(\lambda^{i-j_1})\\
&=J(\chi_0,\chi_0){\sum_{j_1=1}^{2}}\lambda^{-j_1}(-1)g(\lambda^{j_1})g(\lambda^{-j_1})+O(p^{3/2})\\
&=2p^2+O(p^{3/2}).
\end{align*}
This gives
\begin{align*}
A_3^{'}=2p^3+O(p^{5/2}).
\end{align*}
Also 
\begin{align*}
A_3^{''}&=\sum_{\substack{a=2\\a\neq a_1\\a\neq a_2}}^{p-1}\sum_{\substack{c=2\\c\neq a_1\\c\neq a_2}}^{p-1}\sum_{b=1}^{p}\sum_{d=1}^{p}e\left(\frac{x_ab^3-x_cd^3}{p}\right)\\
&=\sum_{\substack{a=2\\a\neq a_1\\a\neq a_2}}^{p-1}\sum_{\substack{c=2\\c\neq a_1\\c\neq a_2}}^{p-1}\sum_{b=1}^{p}\sum_{d=1}^{p}e\left(\frac{x_ab-x_cd}{p}\right)(1+\lambda(b)+\lambda^2(b))(1+\lambda(d)+\lambda^2(d))\\
&=\sum_{\substack{a=2\\a\neq a_1\\a\neq a_2}}^{p-1}\sum_{\substack{c=2\\c\neq a_1\\c\neq a_2}}^{p-1}\sum_{b=1}^{p}\sum_{d=1}^{p}e\left(\frac{b-d}{p}\right)(1+\lambda(\overline{x_a}b)+\lambda^2(\overline{x_a}b))(1+\lambda(\overline{x_c}d)+\lambda^2(\overline{x_c}d))\\
&=\sum_{\substack{a=2\\a\neq a_1\\a\neq a_2}}^{p-1}\sum_{\substack{c=2\\c\neq a_1\\c\neq a_2}}^{p-1}\left(\overline{\lambda}(x_a)g(\lambda)+\overline{\lambda}^2(x_a)g(\lambda^2)\right)\left(\overline{\lambda}(-x_c)g(\lambda)+\overline{\lambda}^2(x_c)g(\lambda^2)\right)\\
&=O(p^2).
\end{align*}
Combining the values for $A_1$, $A_2$ and $A_3$ in $\eqref{p1}$, we complete the proof Theorem \ref{mt2}.

\section{Proof of Theorem \ref{mt3} and \ref{mt4}}
Using Lemma \ref{lem-1} and then applying the estimate in Lemma \ref{lem-2}  and \ref{lem-4}, we deduce 

\begin{align*}
&\sum_{\chi\neq \chi_0}|S_p(\chi)|^2\cdot |L(1,\chi)|\\
&=\sum_{\chi\neq \chi_0}\left(p(1+\chi(a_1)+\chi(a_2))+\sum_{\substack{a=2\\a\neq a_1, a_2}}^{p-1}\chi(a)\left(g(\lambda)\lambda(\overline{x_a})+g(\lambda^{2})\lambda^2(\overline{x_a})\right)\right)\cdot |L(1,\chi)|\\
&=C\cdot (1+C_{a_1}+C_{a_2})\cdot p^2+O(p^{3/2}\ln p).
\end{align*}
Using the fact $a_1a_2\equiv 1 \mod p$, we complete the proof of Theorem \ref{mt3}. Next we write
\begin{align}\label{p7}
&\sum_{\chi\neq \chi_0 }|S_p(\chi)|^4\cdot |L(1,\chi)|\notag\\&=\sum_{\chi\neq \chi_0}\left[p(1+\chi(a_1)+\chi(a_2))+\sum_{\substack{a=2\\a\neq  a_1, a_2}}^{p-1}\chi(a)\sum_{b=1}^{p}e\left(\frac{x_ab^3}{p}\right)\right]^2\cdot |L(1,\chi)|\notag\\
&=B_1+B_2+B_3,
\end{align}
Next we have 
\begin{align*}
B_1&=p^2\sum_{\chi\neq \chi_0 }(1+\chi(a_1)+\chi(a_2))^2\cdot |L(1,\chi)|,\\
B_2&=2p\sum_{\chi\neq \chi_0 }(1+\chi(a_1)+\chi(a_2))\sum_{\substack{a=2\\a\neq a_1,a_2}}^{p-1}\chi(a)\sum_{b=1}^{p}e\left(\frac{x_ab^3}{p}\right)\cdot |L(1,\chi)|,\\
B_3&=\sum_{\chi\neq \chi_0}\left(\sum_{\substack{a=2\\a\neq a_1,a_2}}^{p-1}\chi(a)\sum_{b=1}^{p}e\left(\frac{x_ab^3}{p}\right)\right)^2\cdot |L(1,\chi)|,\\
\end{align*}
Using the estimate in Lemma \ref{lem-4} and \ref{lem-5}, we directly get 
\begin{align}\label{p4}
B_1&=p^2\sum_{\chi\neq \chi_0 }\left(1+2\chi(a_1)+2\chi(a_2)+\chi^2(a_1)+\chi^2(a_2)+2\chi(a_1a_2)\right)\cdot |L(1,\chi)|\notag\\
&=p^3\cdot C\left(3+2(C_{a_1}+C_{a_2})+C_{a_1^2}+C_{a_2^2}\right)+O(p^{5/2}\ln p).
\end{align}
Using Lemma \ref{lem-2} and Weil's estimate \cite{weil} for character sums, we get 
\begin{align}\label{p5}
B_2=O(p^{5/2}).
\end{align}
Now the expression for $B_3$ can be written as 
\begin{align}\label{p2}
B_3&=\sum_{\chi\neq \chi_0}\left(\sum_{\substack{a=1\\a\neq 1, a_1, a_2}}^{p-1}\chi(a)\sum_{b=1}^{p}e\left(\frac{x_ab^3}{p}\right)\right)^2\cdot |L(1,\chi)|\notag\\
&=\sum_{\chi\neq \chi_0}\sum_{\substack{a=2\\a\neq  a_1\\a\neq a_2}}^{p-1}\sum_{\substack{c=2\\c\neq a_1\\c\neq a_2}}^{p-1}\chi(a\overline{c})\sum_{b=1}^{p}\sum_{d=1}^{p}e\left(\frac{x_ab^3-x_cd^3}{p}\right)\cdot |L(1,\chi)|.
\end{align}
Notice that using cubic residue symbol, we have 
\begin{align}\label{p3}
&\sum_{\substack{a=2\\a\neq  a_1\\a\neq a_2}}^{p-1}\sum_{\substack{c=2\\c\neq a_1\\c\neq a_2}}^{p-1}\chi(a\overline{c})\sum_{b=1}^{p}\sum_{d=1}^{p}e\left(\frac{x_ab^3-x_cd^3}{p}\right)\notag\\
&=\sum_{\substack{a=2\\a\neq  a_1\\a\neq a_2}}^{p-1}\sum_{\substack{c=2\\c\neq a_1\\c\neq a_2}}^{p-1}\chi(a)\sum_{b=1}^{p}\sum_{d=1}^{p}e\left(\frac{n(a^3c^3-1)b^3-n(c^3-1)d^3}{p}\right)\notag\\
&=\sum_{\substack{a=2\\a\neq  a_1\\a\neq a_2}}^{p-1}\chi(a)\sum_{\substack{c=2\\c\neq a_1\\c\neq a_2}}^{p-1}\sum_{b=1}^{p}\sum_{d=1}^{p}e\left(\frac{n(a^3c^3-1)b-n(c^3-1)d}{p}\right)\notag\\&\quad\quad\quad\quad\quad\quad\quad\quad\quad\quad\quad(1+\lambda(b)+\lambda^2(b))(1+\lambda(d)+\lambda^2(d)).
\end{align}
Using Weil's \cite{weil} bound for character sums and classical estimate of Gauss sums we get 
\begin{align*}
&\sum_{\substack{c=2\\c\neq a_1, a_2}}^{p-1}\sum_{b=1}^{p}\sum_{d=1}^{p}e\left(\frac{n(a^3c^3-1)b-n(c^3-1)d}{p}\right)(1+\lambda(b)+\lambda^2(b))(1+\lambda(d)+\lambda^2(d))
\\
&=\sum_{c=2}^{p-1}\left(\overline{\lambda}(n(a^3c^3-1))g(\lambda)+\overline{\lambda}^2(n(a^3c^3-1)g(\lambda^2))\right)\\
&\quad\quad\quad\quad\quad\quad\quad\quad\quad\quad\quad\quad\quad\left(\overline{\lambda}(-n(c^3-1))g(\lambda)+\overline{\lambda}^2(-n(c^3-1)g(\lambda^2))\right)+O(p^{3/2})\\
&=O(p^{3/2}).
\end{align*}
Using the above estimate in \eqref{p2} and \eqref{p3}, and then applying Lemma  \ref{lem-2}, we deduce
\begin{align}\label{p6}
B_3=O(p^{5/2}\cdot \ln p).
\end{align} 
Finally, combining the estimate \eqref{p4}, \eqref{p5} and \eqref{p6} in \eqref{p7}, we complete the proof of Theorem \ref{mt4}. 
\section{moment of quadratic Gauss sums}
The question of capturing asymptotic formula for the moment of $k$-generalised Gauss sums seems to be a very difficult. Such questions are addressed, specially for $k=2$ by the author with his collaborators in \cite{BB, BLZ, BLZ1, zhang}. In \cite{zhang}, Zhang conjectured that,
\begin{conjecture}\label{C1}
For all positive integer $m$,
\begin{align*}
\sum_{\chi\neq\chi_0}|G(n,2,\chi;p)|^{2m}\cdot |L(1,\chi)|\sim C\sum_{\chi \bmod p}|G(n,2,\chi;p)|^{2m}, \qquad p\rightarrow+\infty,
\end{align*}
where
\begin{align}\label{constant-c}
C=\prod_p\left[1+\frac{\binom{2}{1}^2}{4^2.p^2}+\frac{\binom{4}{2}^2}{4^4.p^4}+\cdots+\frac{\binom{2m}{m}^2}{4^{2m}.p^{2m}}+\cdots\right]
\end{align}
is a constant and $\displaystyle \prod_p$ denotes the product over all primes and $\chi_0$ denote the principal character modulo $p$.
\end{conjecture} 
 Zhang in the same paper \cite{zhang} studied the hybrid power mean value involving the generalized Quadratic Gauss sums and used estimates for character sums and analytic methods
to study second, fourth and sixth power mean values of generalized quadratic Gauss sums. More study on higher power moments of generalized quadratic Gauss sums can be found in \cite{yuan}. In \cite{BB},the author and Barman proved Conjecture \ref{C1} upto $m\leq 4$, which is subsequently extended to $m=5$ by the author, A.R. Le\'{o}n and W.P. Zhang \cite{BLZ}. Very recently the results in \cite{BLZ1} and along with the results of Xi in \cite{PX} completely established the Conjecture \ref{C1}. 
\par It is very natural to study whether similar asymptotic relation is true for $k>2$. In this article we focus on the case $k=3$. Our results in Theorem 1-4 establish that it is not possible to establish a conjecture analogous to Conjecture \ref{C1} for cubic Gauss sums. Although the author thinks that it is possible to establish such conjecture for even $k$. 

\section{Acknowledgement}
During the prepartion of this article the author was supported by the NBHM post-doctoral fellowship (No.:0204/3/2021/R\&D-II/7363)).



\begin{thebibliography}{99}

%
%



\bibitem{BB}
N. Bag, R. Barman, {\it Higher Order Moments of Generalized Quadratic Gauss Sums Weighted by $L$-functions}, Asian Journal of Mathematics, 25(3) (2021), 413-430.

\bibitem{BLZ}
N. Bag, A. Rojas-Le\'{o}n, W.P. Zhang, {\it An explicit evaluation of 10-th power moment of generalized quadratic Gauss sums and some applications}, Functiones et approximatio, DOI:  10.7169/facm/1995.

\bibitem{BLZ1}
N. Bag, A. Rojas-Le\'{o}n, W.P. Zhang {\it On some conjectures on Generalized quadratic Gauss sums and related problems},  		arXiv:2105.11214 [math.NT]

\bibitem{HB-P}
D. R. Heath-Brown and S. J. Patterson, {\it The distribution of Kummer sums at prime
arguments}, Journal f¨ur die Reine und angewandte Mathematik 310 (1979), 111–130.

\bibitem{cochrane}
T. Cochrane and Z. Y. Zheng, {\it Pure and mixed exponential sums}, Acta Arith. 91 (1999), 249-278.



\bibitem{kummer}
E. E. Kummer, {\it De residuis cubicis disquisitiones nonnullae analyticae}, Journal f¨ur
die Reine und angewandte Mathematik, vol. 32 (1846), 341-359.


\bibitem{pat}
S. J. Patterson, {\it On the distribution of Kummer sums}, J. Reine Angew. Math. 303(304) (1978), 126–
143.


\bibitem{PX}
P. Xi, {\it Moments of certain character sums that are unnamed}, 	arXiv:2105.15051 [math.NT].






\bibitem{yuan}
Y. He and Q. Liao, {\it On an identity associated with Weil's estimate and its applications}, Journal of Number Theory 129 (2009), 1075-1089.
%
%


\bibitem{weil}
A. Weil, {\it On some exponential sums}, Proc. Nat. Acad. Sci. U.S.A. 34 (1948), 203-210.

\bibitem{weil-2}
A. Weil, {\it Basic number theory}, Springer-Verlag, New York, 1974.


\bibitem{zhang}
 W.P. Zhang, {\it Moments of generalized quadratic Gauss sums weighted by $L$-functions},
Journal of Number Theory 92 (2002), 304-314.

\end{thebibliography}
\end{document}